\documentclass[12pt]{amsart}
\usepackage{amssymb,amsmath}

\parskip=\smallskipamount

\newtheorem{theorem}{Theorem}
\newtheorem{lemma}[theorem]{Lemma}

\theoremstyle{definition}
\newtheorem{remark}[theorem]{Remark}

\newcommand{\C}{\mathbb{C}}

\newcommand{\Z}{\mathbb{Z}}

\newcommand{\R}{\mathbb{R}}

\newcommand{\PD}{\operatorname{PD}}
                                 
\newcommand{\Grr}{{\mathrm{Gr}}_\R}
\newcommand{\Grtr}{{\mathrm{Gr}_{tr}}}

\newcommand{\Gl}{{\mathrm{GL}}}
\newcommand{\Sl}{{\mathrm{SL}}}
\newcommand{\Su}{{\mathrm{SU}}}
\newcommand{\U}{{\mathrm{U}}}
\newcommand{\So}{{\mathrm{SO}}}

\newcommand{\cC}{\mathcal{C}}

\newcommand{\cM}{\mathcal{M}}

\newcommand{\cT}{\mathcal{T}}

\newcommand\hra{\hookrightarrow}

\newcommand\Mon{\mathrm{Mon}}

\numberwithin{equation}{section}

\makeatletter
\@namedef{subjclassname@2020}{\textup{2020} Mathematics Subject Classification}
\makeatother

\begin{document}
\title[Cancelling CR singularities of $3$-manifolds in complex threefolds]
{Cancelling CR singularities of 3-manifolds in complex threefolds}
\author[Marko Slapar]{Marko Slapar}
\address{Marko Slapar, Faculty of Mathematics and Physics\\ University of Ljubljana \\ Jadranska 19\\1000 Ljubljana, Slovenia\\ \newline
Faculty of Education\\ University of Ljubljana \\ Kardeljeva plo\v{s}\v{c}ad 16\\1000 Ljubljana, Slovenia\\
and\newline
  Institute of Mathematics, Physics and Mechanics\\Jadranska
  19\\1000 Ljubljana, Slovenia}
\email{marko.slapar@fmf.uni-lj.si}

\author[Sa\v{s}o Strle]{Sa\v{s}o Strle}
\address{Sa\v{s}o Strle, Faculty of Mathematics and Physics \\ University of Ljubljana \\ Jadranska 21 \\ 1000 Ljubljana, Slovenia }
\email{saso.strle@fmf.uni-lj.si}
\date{\today}

%
%
\thanks{The first author was supported by the European Union through the ERC Advanced Grant HPDR, grant agreement No. 101053085, 
awarded to Franc Forstneri\v{c}, and by the research program P1-0291 and research project J1-70033 of ARIS, Republic of Slovenia.
The second author was supported by the research program P1-0288 of ARIS, Republic of Slovenia.}

\begin{abstract}
Let $M$ be a closed oriented $3$-manifold generically embedded in a complex $3$-manifold $X$. Its CR singular set is an 
oriented link $L\subset M$. We prove that if a sublink $L'\subset L$ bounds an oriented Seifert surface $S\subset M$ 
in the complement of $L\setminus L'$, then the CR singularities along $L'$ can be cancelled by an arbitrarily 
$\cC^0$-small isotopy supported in an arbitrarily small neighbourhood of $S$.\end{abstract}
\maketitle

\section{Introduction} 
Let $M$ be a closed oriented $n$-manifold embedded in an $n$-dimensional complex manifold $X$. The complex tangent space 
at a point $p\in M$ is defined as $T_p^\C M=T_pM\cap JT_pM$, where we consider $TM$ as a real subbundle of $TX$  
and $J$ is the complex structure on $X$. The embedding is totally real at $p$ if $T_p^\C M=\{0\}$, whereas
$p$ is called a CR singular point if $\dim_\C T_p^\C M>0$.  If the embedding is sufficiently generic, then CR singular 
points of $M$ form a submanifold $L$ of codimension $2$ in $M$, and $\dim_\C T_p^\C M=1$ for $p\in L$ \cite{Lai}. 
CR singular points can be further classified as elliptic, hyperbolic or parabolic. At least in the case where 
$\dim M\le 5$, and under additional genericity assumptions on the embedding, certain topological counts can be defined, 
connecting the Euler class of $M$ and the normal bundle of $M$ with the Euler characteristics of elliptic and hyperbolic 
points and the so-called parabolic index at the submanifold of parabolic points \cite{W}.

The theory of CR singular points is well understood in dimension $n=2$, where the CR singular set is zero-dimensional 
and parabolic points are absent for generic embeddings. If $M$ is an oriented surface in a complex surface $X$, elliptic 
and hyperbolic points can also be assigned an orientation sign, depending on whether the orientation of the tangent space 
of $M$ agrees with the induced orientation coming from $X$. In this situation, a pair consisting of an elliptic point and 
a hyperbolic point with the same orientation sign can be cancelled by a $\cC^0$-small isotopy, supported in a small tubular 
neighbourhood of an arc whose endpoints are the two points and whose interior lies in the totally real part \cite{EH,F1}. 

Every oriented closed $3$-manifold is parallelizable and in the case of $X=\C^3$, every embedding of a closed oriented 
$3$-manifold into $\mathbb C^3$ is isotopic to a totally real embedding \cite{Gr,F0}. On the other hand, Kasuya and Takase 
\cite{KT} have shown that any null-homologous link $L\subset M$ can be realised precisely 
as the link of CR singular points for an appropriate embedding of $M$ into $\C^3$. Cancellation of isolated CR singular 
points was treated in \cite{E1}, and a version of the cancellation theorem for embeddings of $S^3$ into $\C^3$ was proved 
in \cite{E2}.

From now on, let $X$ be a complex $3$-manifold with complex structure $J$. We call an embedding of an oriented $3$-manifold $M$ into $X$ \emph{generic} if the Gauss map of the embedding is transverse to
\[\Grr^+(3,TX)\setminus \Grtr(3,TX),\]
where $\Grr^+(3,TX)\to X$  denotes the Grassmannian bundle of oriented real $3$-planes in $TX$, and $\Grtr(3,TX)\subset \Grr^+(3,TX)$ is the open subbundle consisting of those oriented real $3$-planes $\Pi\subset T_xX$ satisfying $
\Pi\cap J\Pi=\{0\}.$ The orientation of $L$ is then the one induced by viewing $L$ as the transverse zero set of the determinant section
\[df(e_1)\wedge_{\C}df(e_2)\wedge_{\C}df(e_3)
\in \Gamma\left(M,\Lambda_{\C}^3 f^*TX\right),\]
where $(e_1,e_2,e_3)$ is a positive frame of $TM$. This orientation is independent of the chosen local positive frame.
In higher dimensions, an embedding is usually called generic if it is also transverse to higher-order degeneracy strata in $\Grr^+(n,TX)$, see \cite{Lai}. The main result of this paper is the following cancellation theorem for CR singular points. 

\begin{theorem}\label{thm1}
Let $M$ be a closed oriented smooth $3$-manifold and let $f:M\hookrightarrow X$ be a generic embedding into a complex $3$-manifold. Let $L\subset M$ be the oriented link of CR singular points of $f$. Let $L'\subset L$ be a sublink. Suppose that $L'$ bounds an oriented embedded Seifert surface $S\subset M$ such that
\[S\cap (L\setminus L')=\emptyset.\] 
Let $U$ be a sufficiently small regular neighbourhood of $S$ with
\[U\cap L=L'\quad \text{and}\quad \partial U\cap L=\emptyset.\]
Then, for every $\varepsilon>0$ and every choice of metric $d$ on $X$, there exists an isotopy of embeddings
\[f_t:M\hookrightarrow X,\qquad t\in[0,1],\]
starting with $f_0=f$ and such that for all $t\in[0,1]$
\[f_t=f \text{ on } M\setminus U,\qquad \sup_{p\in M}d(f_t(p),f(p))<\varepsilon,\]
and the CR singular link of $f_1$ is precisely $L\setminus L'$.
\end{theorem}

\begin{remark} 
The integrability of the complex structure $J$ is not used in the proof. Hence Theorem \ref{thm1} remains valid for 
$J$-complex tangent points of closed oriented $3$-manifolds generically embedded in an almost complex $6$-manifold $(X,J)$.
\end{remark}

For $X=\C^3$, the link $L$ of CR singular points of $f:M\hra \C^3$ is the zero set of the complex-valued function
\[\Delta_{df} = df(e_1)\wedge_{\mathbb C}df(e_2)\wedge_{\mathbb C}df(e_3),\]
where $(e_1,e_2,e_3)$ is a chosen positive global frame of $TM$. Then the preimage of a regular value of the normalization of $\Delta_f$ on $M \setminus L$ gives a Seifert surface for $L$. Hence $L$ is null-homologous in $M$ and Theorem \ref{thm1} implies that $f$ can be isotoped to a totally real embedding.

For a generic embedding $f:M\hra X$, the full CR singular link need not be null-homologous. With the orientation induced as the zero set of the determinant section, one has
\[[L]=\PD_M\left(f^*c_1(TX)\right)\in H_1(M;\Z).\]
Using Theorem \ref{thm1} one can cancel a sublink $L'\subset L$ which bounds an oriented embedded Seifert surface in the chosen neighbourhood $U$. In particular, $[L']=0$ in $H_1(M;\Z)$. Thus, with the induced orientations, 
\[[L\setminus L']=[L]\in H_1(M;\Z).\]
On the other hand, if
\[f^*c_1(TX)=0 \text{\ in\ } H^2(M;\Z),\]
then $[L]=0$ in $H_1(M;\Z)$. Hence the oriented link $L$ bounds an oriented embedded Seifert surface $S\subset M$. Applying Theorem \ref{thm1} with $L'=L$, we obtain the following result.

\begin{theorem}
Let $f:M\hra X$ be a generic embedding of a closed oriented $3$-manifold into a complex $3$-manifold $X$. Suppose that
\[f^*c_1(TX)=0\text{\ in\ } H^2(M;\Z).\]
Then $f$ can be isotoped to a totally real embedding by an arbitrarily $\cC^0$-small isotopy. In particular, the conclusion holds if $X=\C^3$ or if $M$ is an integral homology 3-sphere.  
\end{theorem}

\section{Proof of Theorem \ref{thm1}}
The proof of the theorem first deforms the differential $df$ on $U$, relative to the boundary of $U$, to a totally real formal differential, and then applies the relative h-principle for totally real embeddings.

 We first prove the following lemma. Recall that every oriented $3$-manifold is parallelizable.

\begin{lemma}\label{lemma}
Let $N$ be a compact oriented smooth $3$-manifold with boundary and let $E\to N$ be a complex vector bundle of rank $3$. Denote by 
\[\cM=\Mon_{\R}(TN,E)\] 
the bundle of $\R$-linear monomorphisms and let $\cT\subset \cM$ be the open subbundle consisting of totally real monomorphisms. Let
\[A\in \Gamma(N,\mathcal M)\]
be a section which takes values in $\cT$ over $\partial N$.  Choose a positive global frame $(e_1,e_2,e_3)$ of $TN$, and define
\[\Delta_A=A(e_1)\wedge_{\C} A(e_2)\wedge_{\C} A(e_3)
        \in \Gamma\left(N,\Lambda^3_{\C}E\right).\]
Assume that $\Delta_A$ is transverse to the zero section and let
\[Z_A=\Delta_A^{-1}(0),\]
with the orientation induced as the zero set of the section $\Delta_A$ of the complex line bundle $\Lambda^3_{\C}E$. 
Then $A$ is homotopic relative to $\partial N$, through sections of $\mathcal M$, to a section of $\mathcal T$ if and only if 
\[[Z_A]=0 \text{\ in\ } H_1(N;\Z).\]
Moreover, the homotopy may be chosen fixed on a sufficiently small collar of $\partial N$.
\end{lemma}

\begin{proof}
A real monomorphism \(B:T_pN\to E_p\) is totally real if and only if 
\[B(e_1)\wedge_{\C}B(e_2)\wedge_{\C}B(e_3)\ne 0.\] 
Thus the subbundle \(\cT\subset \cM\) is exactly the inverse image of the complement of the zero section under the determinant map. If a positive framing of $TN$ is changed by a map to $\Gl^+(3,\R)$, then $\Delta_A$ is multiplied by the positive real determinant of this change of frame. Hence the zero set $Z_A$, its induced orientation, and homology class $[Z_A]$ 
are independent of the chosen positive framing.

We apply obstruction theory to the bundle pair
\[(\cM,\cT)\longrightarrow N.\]
After choosing local positive frames of $TN$ and local complex frames of $E$, the fiber pair is identified with 
\[(V,W)=\left(\Mon_{\R}(\R^3,\C^3),\Gl(3,\C)\right).\]
The fiber $V$ deformation retracts to the real Stiefel manifold $V_{6,3}$, and $W$ deformation retracts to $\U(3)$. 
Since $V_{6,3}$ is $2$-connected, the relative homotopy exact sequence 
\begin{equation*}	
\begin{split}	
\cdots \longrightarrow
\pi_3(U(3))
\longrightarrow
\pi_3(V_{6,3})
\longrightarrow
\pi_3(V_{6,3},U(3))
\overset{\partial}{\longrightarrow}
\pi_2(U(3))
\longrightarrow
\pi_2(V_{6,3})\\
\longrightarrow
\pi_2(V_{6,3},U(3))
\overset{\partial}{\longrightarrow}
\pi_1(U(3))
\longrightarrow
\pi_1(V_{6,3})
\longrightarrow
\pi_1(V_{6,3},U(3))
\longrightarrow
0
\end{split}
\end{equation*}
gives
\[\pi_1(V,W)=0, \qquad \pi_2(V,W)\cong \pi_1(W)\cong \pi_1(\U(3))\cong \Z.\]
Under the connecting isomorphism $\pi_2(V,W)\cong \pi_1(W)$, the generator is detected by the determinant. Indeed,
\[\det\nolimits_{\C}:W=\Gl(3,\C)\longrightarrow \C^*\]
induces an isomorphism on $\pi_1$, since $\Sl (3,\C)\simeq \Su(3)$ is simply connected. Equivalently, the preferred generator of \(\pi_2(V,W)\) is represented by a small positively oriented normal disk to the codimension-two determinant locus $\{\det\nolimits_{\C}=0\}\subset V$; its boundary is a positive meridian.

There is no secondary obstruction in degree $3$. Indeed, the map 
\[\pi_3(W)\longrightarrow \pi_3(V)\]
is surjective. To see this, write $V\simeq \operatorname{SO}(6)/\So(3)$. The standard map $\U(3)\to \operatorname{SO}(6)$, obtained by forgetting the complex structure, is an isomorphism on $\pi_3$; for instance, this follows from the fibration
\[\U(3)\longrightarrow \operatorname{SO}(6)\longrightarrow \So(6)/U(3)
        \cong \C P^3\]
and the fact that $\pi_k(\mathbb{C}P^3)=0$ for $k=3,4$. The projection $\So(6)\to \So(6)/\So(3)$ is surjective on $\pi_3$, by the fibration
\[\So(3)\longrightarrow \So(6)\longrightarrow \So(6)/\So(3)\]
and $\pi_2(\operatorname{SO}(3))=0$. Hence
\[\pi_3(V,W)=0 .\]

We apply obstruction theory to address the question of when it is possible to deform $A$, relative to $\partial N$, through sections of $\cM$ to a section of $\cT$. Fixing a structure of a CW complex on $N$ it follows, using the triviality of $\pi_k(V,W)$ for $k=0, 1$, that $A$ is homotopic to a section that maps the 1-skeleton into $\cT$. Since for all $x\in N$, $\pi_2(\cM_x,\cT_x) \cong \pi_2(V,W) \cong \Z$, there is an obstruction class in 
\[H^2\left(N,\partial N;\underline{\pi_2(\cM_x,\cT_x)}\right)\]
whose vanishing is equivalent to existence of a section homotopic to $A$ that maps the 2-skeleton into $\cT$. If this obstruction vanishes, it follows from $\pi_3(V,W) = 0$ that $A$ can be homotoped, relative to $\partial N$, to a section of $\cT$.

Note that the local coefficient system $\underline{\pi_2(\cM_x,\cT_x)}$ is trivial. The determinant locus in $\cM_x$ is independent of the choice of trivializations of the bundles and hence defines a global codimension-two locus in $\cM$. Transporting the generator of $\pi_2(\cM_x,\cT_x)$, represented by a positively oriented disk normal to the determinant locus, around any loop will result in a positively oriented normal disk representing the same generator. This follows, since a change of positive frame in $TN$ and a change of complex frame in $E$ preserves the orientations. Alternatively, mapping the fibre pair by the determinant homomorphism $\det_{\C} : (V,W) \to (\C,\C^*)$ induces isomorphism on $\pi_2$. A change in frames multiplies the determinant by a nowhere-zero complex-valued function and hence preserves the generator of $\pi_2(\C,\C^*)$, showing that the local coefficient group $\underline{\pi_2\left(\Lambda^3_{\C}E_x,\Lambda^3_{\C}E_x\setminus\{0\}\right)}$ is trivial. It follows that the only obstruction to deforming $A$, relative to $\partial N$, through sections of $\cM$ into $\cT$ lies in
\[H^2(N,\partial N;\Z).\]

Next we identify this obstruction. By naturality of obstruction theory, the determinant map sends the primary obstruction for the pair $(\cM,\cT)$ to the primary obstruction for the pair
\[\left(\Lambda^3_{\C}E,\Lambda^3_{\C}E\setminus\{0\}\right).\]
Since the induced map
\[(\det\nolimits_\C)_*:\pi_2(V,W)\longrightarrow \pi_2(\mathbb C,\mathbb C^*)\]
is an isomorphism, these obstruction classes are identified. The latter obstruction is precisely the obstruction to deforming the determinant section $\Delta_A$, fixed on $\partial N$, to a nowhere-zero section of the complex line bundle $\Lambda^3_{\mathbb C}E$. Hence it is the relative first Chern class
\[c_1\left(\Lambda^3_{\mathbb C}E,\Delta_A|_{\partial N}\right)\in H^2(N,\partial N;\mathbb Z),\]
where the nonzero boundary section \(\Delta_A|_{\partial N}\) is used as the boundary trivialization. Since $\Delta_A$ is transverse to the zero section, the usual zero-locus representative of the relative Chern class gives
\[c_1\left(\Lambda^3_{\C}E,\Delta_A|_{\partial N}\right) =\PD_N[Z_A]\in H^2(N,\partial N;\Z),\]
with $Z_A$ oriented as the zero set of the section $\Delta_A$ of the complex line bundle $\Lambda^3_{\C}E$. Hence the obstruction vanishes if and only if
\[ [Z_A]=0.\]

Finally, since $\partial N$ is compact and $\cT$ is open, $A$ takes values in $\cT$ on a (smooth) collar of $\partial N$. 
Applying the same relative obstruction argument to the complement of this collar and keeping the new boundary fixed, and then extending the resulting homotopy by the constant homotopy on the collar, gives a homotopy fixed on a collar of $\partial N$.
\end{proof}

We now take a generic embedding $f:M\hra X$ and apply the preceding lemma with 
\[N=U,\qquad E=f^*TX|_U,\qquad A=df|_U.\]
Choose a positive global frame $(e_1,e_2,e_3)$ of $TU$ and define
\[\Delta=\Delta_{df}=df(e_1)\wedge_{\C}df(e_2)\wedge_{\C}df(e_3)
       \in \Gamma(U,\Lambda^3_{\C}f^*TX|_U).\]
By genericity, $\Delta$ is transverse to the zero section and
\[ \Delta^{-1}(0)=L'.\] 
Moreover, $\Delta$ is nonzero on $\partial U$, since $f$ is totally real there. The obstruction to homotoping the real differential $df|_U$, relative to $\partial U$, to a totally real section is the class
\[[\Delta^{-1}(0)] = [L'] \in H_1(U;\Z).\]
 Since $L'=\partial S$ in $U$, this class vanishes.
 Hence, by Lemma \ref{lemma}, there exists a homotopy of sections
\[A_t\in \Gamma(U,\cM),\qquad t\in[0,1],
\]
fixed near $\partial U$, such that
\[A_0=df|_U,\qquad A_1\in \Gamma(U,\mathcal T),\]
where now
\[\cM=\Mon_{\R}(TU,f^*TX|_U),\qquad \cT\subset \mathcal M
\]
is the subbundle of totally real monomorphisms.
 By the relative $h$-principle for totally real embeddings \cite[Section 2.4.5]{Gr}, \cite[Section 19.4]{EM}, this formal homotopy, fixed near $\partial U$, gives an arbitrarily $\cC^0$-small isotopy supported in $U$.
 After this isotopy, the embedding is totally real on $U$. Since $f$ is unchanged on $M\setminus U$, and since
\[
U\cap L=L',
\]
the CR singular link of the final embedding is precisely $L\setminus L'$. This proves the theorem.
\begin{remark}
It is worth noting that, in dimension $3$, the cancellation problem has no additional obstruction coming from the local type of the CR singularities, such as ellipticity, hyperbolicity, or parabolicity. Once the relative determinant obstruction vanishes, the obstruction theory for the pair $(V_{6,3},U(3))$ gives a homotopy of the differential to the totally real locus. The possible secondary obstruction vanishes because $\pi_3(V_{6,3},U(3))=0.$ We expect the situation to be quite different in higher dimensions: even when the analogous determinant  obstruction vanishes, further obstruction-theoretic difficulties may occur in deforming the full differential to a totally real formal differential. For example, when $n=4$, after the determinant obstruction has vanished, there may be an additional obstruction with coefficient group $\pi_4(V_{8,4},U(4))\cong \Z\oplus\Z$. 
\end{remark}

\bibliographystyle{amsplain}

\end{document}